\documentclass[12pt,reqno,a4paper]{amsart}
\usepackage{mathrsfs}
\usepackage{amsfonts}
\usepackage{latexsym,amsmath,amsfonts,amssymb,amsthm}
\theoremstyle{plain}
\newtheorem{Thm}{Theorem}
\newtheorem{Lem}{Lemma}

\theoremstyle{definition}
\newtheorem*{Ack}{Acknowledgment}
\theoremstyle{remark}

\def\Z{\mathbb Z}
\def\N{\mathbb N}

\def\c{{\bf c}}
\def\cP{\mathcal P}

\def\cA{\mathcal{A}}
\def\cB{\mathcal{B}}
\def\E{\mathbb{E}}
\def\P{\mathbb{P}}

\def\1{{\bf 1}}
\def\pmod #1{\ ({\rm mod}\ #1)}

\def\floor #1{\lfloor{#1}\rfloor}

\begin{document}
\title{The additive complements of primes and Goldbach's problem}
\author{Li-Xia Dai}
\email{lilidainjnu@163.com}
\address{School of Mathematical Sciences, Nanjing Normal University, Nanjing  210046, People's Republic of China}
\author{Hao Pan}
\email{haopan79@yahoo.com.cn}
\address{Department of Mathematics, Nanjing University,
Nanjing 210093, People's Republic of China}
\thanks{The first author is
supported by National Natural Science Foundation of China (Grant
No. 10801075) and the Natural Science Foundation of Jiangsu Higher
Education Institutions of China (Grant No.08KJB11007). The second
author is supported by National Natural Science Foundation of
China (Grant No. 10901078).}

 \maketitle

\section{Introduction}
\setcounter{Lem}{0}\setcounter{Thm}{0}\setcounter{Cor}{0}
\setcounter{equation}{0}

First, we introduce the notion of additive complements. For a set $A\subseteq\N=\{0,1,2,3,\ldots\}$, we say a set
$B\subseteq\N$ is an additive complement of $A$, provided that for
every sufficiently large $n\in\N$, there exist $a\in A$ and $b\in B$ such that
$n=a+b$, i.e., the sumset
$$
A+B=\{a+b:\,a\in A,\ b\in B\}
$$
contains all sufficiently large integers.

Furthermore, for the sets $A,B\subseteq\N$, if the sumset
$A+B$ has lower density 1, i.e., almost all positive integers $n$ can be represented as $n=a+b$ with $a\in A$ and $n\in B$,
then we say $B$ is an almost additive complement of $A$.

The additive properties of primes are always one of the most fascinating topics in number theory. Let $\cP$ denote the set of all primes. It is natural to ask what the additive complements of $\cP$ are. By the prime number theorem, we know that
$$
\cP(x)\sim\frac{x}{\log x},
$$
where $A(x)=|A\cap[1,x]|$ for a set $A\subseteq\N$. So if $A$ is an additive complement of $\cP$, we must have
$A(x)\gg\log x$. Unfortunately, no one knows whether there exists an additive complement $A$ of $\cP$
satisfying $A(x)=O(\log x)$. However,
in \cite{Erdos54}, using the probability method, Erd\H
os showed that such additive complement exists, provided that we replace $\log x$ by $(\log x)^2$. That is,

\medskip{\it there exists a set $A\subseteq\N$ with $A(x)=O((\log x)^2)$ such that the sumset $A+\cP$ contains every sufficiently large integer.}\medskip

In 1998, Ruzsa \cite{Ruzsa98} improved the results of Wolke
\cite{Wolke96} and Kolountzakis\cite{Kolountzakis96}, and showed
that for every function $w(x)$ with
$\lim_{x\to+\infty}w(x)=+\infty$, there exists an almost additive
complement $A$ of $\cP$ with $A(x)=O(w(x)\log x)$, i.e.,

\medskip{\it there exists a set $A\subseteq\N$ with $A(x)=O(w(x)\log x)$ such that almost all positive integers can be represented as the sums of one element of $A$ and a prime.}\medskip

 In 2001, Vu proved that $\cP$ has an additive complement $A$ of order 2 with $A(x)=O(\log x)$, i.e.,

 \medskip{\it there exists a set $A\subseteq\N$ with $A(x)=O(\log x)$ such that every sufficiently large integer $n$ can be represented as $n=a_1+a_2+p$, where $a_1,a_2\in A$ and $p$ is a prime.}\medskip

Clearly Vu's result implies Erd\H os' result, since $(A+A)(x)\ll (A(x))^2$.

Next, let us turn to the Goldbach problem. As early as 1937, using the circle method, Vinogradov \cite{Vinogradov37} had solved the ternary Goldbach problem and showed that

\medskip{\it every sufficiently large odd integer can be represented as the sum of three primes.}\medskip

And subsequently, with a similar discussion, Estermann \cite{Estermann38} proved the binary Goldbach problem is true for almost all positive even integers, i.e.,

\medskip{\it almost all positive even integers can be represented as the sums of two primes.}\medskip

In this note, we shall combine the results of Ruzsa and Vu with Goldbach's problem.
\begin{Thm}\label{t1}
There exists a set $\cA\subseteq\cP$ with $\cA(x)=O(\log x)$ such that every sufficiently large odd integer can be represented as $a_1+a_2+p$ where $a_1,a_2\in\cA$ and $p\in\cP$.
\end{Thm}
\begin{Thm}\label{t2}
For every function $w(x)$ with $\lim_{x\to+\infty}w(x)=+\infty$, there exists a set $\cB\subseteq\cP$ with $\cB(x)=O(w(x)\log x)$ such that almost all even positive integers can be represented as $b+p$ where $b\in\cB$ and $p\in\cP$.
\end{Thm}
Note that Theorems \ref{t1} and \ref{t2} imply Vu's and Ruzsa's results respectively. If we set $A=\cA\cup(\cA+\{1\})$ and $B=\cB\cup(\cB+\{1\})$, then clearly the sumset $A+A+\cP$ contains all sufficiently large integer and $B+\cP$ has lower density $1$.

The proofs of Theorems \ref{t1} and \ref{t2} will be given in the next sections.

\section{Proof of the Theorem \ref{t1}}
\setcounter{equation}{0} \setcounter{Thm}{0} \setcounter{Lem}{0}
\setcounter{Cor}{0}

The key of our proofs is the following lemma.
\begin{Lem}\label{jia} There exists a positive constant $c_0$ such that
if $x$ is sufficiently large, $x^{1-c_0}\leq M\leq x$ and $0\leq y\leq
x-M$, then for all even integers $n$ with $x\leq n\leq x+M$,
except for $O(M(\log x)^{-2})$ exceptional values, we always have
\begin{equation}\label{jiae}
\sum_{\substack{ n=p_1+p_2\\ y\leq p_1 \leq y+M\\x-y-M\leq p_2\leq
x-y+M }}1\gg C(n)\frac{M}{(\log x)^2}
\end{equation}
where
\begin{equation}\label{cn}
C(n)=\prod_{\substack{p\nmid n }}\bigg(1-\frac{1}{(p-1)^2}\bigg)\prod_{\substack{p\mid n }}\bigg(1+\frac{1}{p-1}\bigg)
\end{equation}
and the implied constant in (\ref{jiae}) only depends on $c_0$.
\end{Lem}
\begin{proof} This lemma can be proved by the method of Jia in \cite{Jia94}, although he only discussed the case $y=x/2$.
In fact, Jia proved that Lemma \ref{jia} holds whenever $c_0>7/12$.
\end{proof}
Now suppose that $n$ is a sufficiently integer. For each $x\in\cP$, we choose $x$ to be in $\cA$ with the probability $$
\varrho_x=\frac{\c\log x}x,$$
where $\c>0$ is a constant to be chosen
later. Then using the law of large numbers, one can easily show that almost surely $\cA(n)=O(\log n)$ for every $n$, since
$$\sum_{\substack{p\leq n\\ p\text{ is prime}}}\frac{\log p}{p}\sim\log n.
$$
Let $t_x$ be the binary random
variable representing the choice of $x$, i.e., $t_x=1$ with
probability $\varrho_x$ and 0 with probability $1-\varrho_x$. Consider
$$Y_n=\sum_{\substack{p<n\\ p\ \text{is prime}}}\sum_{\substack{i+j=n-p\\ i,j\ \text{are
prime}}}t_it_j.$$
We need to prove that there exists $n_0>0$ satisfying the probability
$$\P(Y_n>0\text{ for every }n\geq n_0)\geq 1/2.$$ Choose $0<\epsilon<c_0/2$ and let $M=n^{1-2\epsilon}$, where $c_0$ is the one appearing in Lemma \ref{jia}. In the remainder of this section, the implied constants of $O(\cdot)$, $\ll$ and $\gg$ will only depend on $\epsilon$. Let
$$Y_n^*=\sum_{\substack{p\leq n-n^{1-\epsilon}\\ p\ \text{is prime}}}\sum_{\substack{i+j=n-p\\ i,j \geq M \\i,j\text{ are prime}}}t_it_j.$$
Clearly $\P(Y_n>0)\geq\P(Y_n^*>0)$ for every $n$. The following lemma is due to Janson \cite{Janson90}.
\begin{Lem}\label{janson}
Let $z_1,\cdots,z_m $ be independent indicator random variables
and $Y=\Sigma_\alpha I_\alpha$ where each $I_\alpha$ is the
product of few $z_j's.$ Define $\alpha\sim\beta$ if there is some
$z_j$ which is in both $I_\alpha$ and $I_\beta$. Let
$\Delta=\sum_{\alpha\sim\beta}\E(I_\alpha I_\beta)$. For any $Y$
and any positive number $\varepsilon$ we have
$$ \P(Y\leq
(1-\varepsilon)\E(Y))\leq \exp\bigg(-\frac{(\varepsilon
\E(Y))^2}{2(\E(Y)+\Delta )}\bigg).
$$
\end{Lem}
Let
$$
X=\{(i,j):\,i,j\geq M,\ i+j\geq n^{1-\epsilon}\text{ and
}i,j,n-i-j\text{ are all prime}\}.$$ Clearly,
$$
Y_n^*=\sum_{(i,j)\in X}t_it_j=\sum_{\alpha\in X}I_\alpha,
$$
where $I_\alpha=t_it_j$ for $\alpha=(i,j)$. In view of Lemma \ref{janson}
we only need to estimate
$$
\E(Y_n^*)=\E\bigg(\sum_{\alpha \in
X}I_\alpha\bigg)
\text{\qquad and\qquad}
\Delta=\E\bigg(\sum_{\substack{\alpha,\beta\in X\\
\alpha\sim\beta}}I_\alpha I_{\beta}\bigg).
$$
By Lemma \ref{jia}, we have
\begin{align*}
\E\bigg(\sum_{\alpha \in X}I_\alpha\bigg)=&\E\bigg(\sum_{\substack{p\leq
n-n^{1-\epsilon}\\ p\text{ is prime}}}
\sum_{\substack{i+j=n-p\\ i,j\geq M\\ i,j\text{ are prime}} }t_it_j\bigg)\geq\E\bigg(\sum_{\substack{0\leq t\leq n^{1-\epsilon}/M \\
tM\leq n-n^{1-\epsilon}-p<(t+1)M\\ p\text{ is
prime}}}\sum_{\substack{ 1\leq s\leq
(n-p)/M-2\\ i+j=n-p\\ sM\leq i<(s+1)M\\ i,j\text{ are prime}}}t_it_j\bigg)\\
\geq&\E\bigg(\sum_{\substack{0\leq t\leq n^{1-\epsilon}/M\\ 1\leq s\leq n^{1-\epsilon}/M+t-2}}\sum_{\substack{tM\leq n-n^{1-\epsilon}-p<(t+1)M\\
 p\text{ is prime}}}\sum_{\substack{i+j=n-p\\ sM\leq i<(s+1)M\\ i,j\text{ are prime}}}t_it_j\bigg)\\
\gg&\sum_{\substack{0\leq t\leq n^{1-\epsilon}/M\\ 1\leq s\leq
n^{1-\epsilon}/M+t-2}} \frac{M}{\log M}\cdot\frac{M}{(\log
M)^2}\cdot\frac{\c^2(\log
M)^2}{sM\cdot(n^{1-\epsilon}+(t+1)M-sM)}.
\end{align*}
Clearly,
\begin{align*}
&\sum_{\substack{0\leq t\leq n^{1-\epsilon}/M\\ 1\leq s\leq
n^{1-\epsilon}/M+t-2}}
\frac{1}{sM\cdot(n^{1-\epsilon}+(t+1)M-sM)}\\
\gg&\int_{1}^{n^{1-\epsilon}/M}\bigg(\int_{1}^{n^{1-\epsilon}/M+t-2}\frac{1}{sM(n^{1-\epsilon}+tM-sM)}d s\bigg)d t\\
\gg&\int_{1}^{n^{1-\epsilon}/M}\frac{1}{n^{1-\epsilon}+tM}\bigg(\int_{1}^{n^{1-\epsilon}/M+t-2}
\bigg(\frac{1}{sM}+\frac1{n^{1-\epsilon}+tM-sM}\bigg)d s\bigg)d t\\
\gg&\int_{1}^{n^{1-\epsilon}/M}\frac{\log(n^{1-\epsilon}+tM)}{M(n^{1-\epsilon}+tM)}d t\gg \frac{(\log(2n^{1-\epsilon}))^2}{M^2}\gg \frac{(\log n)^2}{M^2}.\\
\end{align*}
Thus we get that
$$\E(Y_n^*)\gg \c^2\log n.$$
Now we turn to $\Delta$.
\begin{align*}
\E\bigg(\sum_{\substack{\alpha,\beta\in
X\\\alpha\sim\beta}}I_\alpha
I_{\beta}\bigg)=&\E\bigg(\sum_{\substack{i\geq M\\ i\text{ is
prime}}}
\sum_{\substack{p_1+j_1=p_2+j_2=n-i\\ p_1,p_2\leq n-n^{1-\epsilon}\\ j_1,j_2\geq M\\ p_1,p_2,j_1,j_2\text{ are prime}} }(t_it_{j_1})(t_it_{j_2})\bigg)\\
\leq&\sum_{\substack{i\geq M\\ i\text{ is prime}}}\frac{\c\log i}{i}\sum_{\substack{p_1+j_1=p_2+j_2=n-i\\
p_1,p_2\leq n-n^{1-\epsilon}\\ j_1,j_2\geq M\\
p_1,p_2,j_1,j_2\text{ are prime}} } \frac{\c\log
j_1}{j_1}\cdot\frac{\c\log j_2}{j_2}\\\leq&\sum_{\substack{M\leq
i\leq n-n^{1-\epsilon}\\ i\text{ is prime}}} \frac{\c\log
n}{i}\bigg(\sum_{\substack{p+j=n-i\\ j\geq M\\ p,j\text{ are
prime}} }\frac{\c\log n}{j}\bigg)^2.
\end{align*}

For a set $U$ of positive integers, by the partial summation, we
have
$$
\sum_{\substack{M\leq j\leq x\\ j\in
U}}\frac{1}j\ll\frac{U(x)}{x}-\frac{U(M)}{M}+\sum_{M<y\leq
x}\frac{U(y)}{y^2}.
$$
Note that by the sieve method, we have
$$
|M\leq j\leq y:\,\text{both }j\text{ and }m-j\text{ are
primes}|\ll\frac{C(m)(y-M)}{(\log(y-M))^2}\ll\frac{C(m)y}{(\log
y)^2}
$$
since $y/(\log y)^2$ is increasing for $y>e$, where $C(\cdot)$ is the one appearing in Lemma \ref{jia}. Hence
\begin{align*}
\sum_{\substack{p+j=n-i\\ j\geq M \\ p,j\text{ are prime}}
}\frac{1}{j}\ll&\frac{1}{(n-i)}\cdot\frac{C(n-i)\cdot(n-i)}{(\log(n-i))^2}+\frac1M+\sum_{M+e<y\leq
n-i}\frac{C(n-i)}{y(\log y)^2}\\\ll&\frac{C(n-i)}{\log
M}\ll\frac{C(n-i)}{\log n}.
\end{align*}
Thus,
$$
\E\bigg(\sum_{\alpha\sim\beta}I_\alpha
I_{\beta}\bigg)\ll \c^3\log n\sum_{\substack{M\leq i\leq n-n^{1-\epsilon}\\
i\text{ is prime}}}\frac{C(n-i)^2}{i}.
$$
Note that
$$
C(n-i)^2\ll\prod_{p\mid
n-i}\bigg(1+\frac{1}{p-1}\bigg)^2\ll\prod_{p\mid
n-i}\bigg(1+\frac{1}{p}\bigg)^2\ll\prod_{p\mid
n-i}\bigg(1+\frac{2}{p}\bigg)=\sum_{\substack{d\mid n-i\\ d\text{
is square-free}}}\frac{2^{\omega(d)}}{d},
$$
where $\omega(d)$ denotes the number of the distinct prime factors
of $d$. Hence,
$$
\sum_{\substack{M\leq i\leq n-n^{1-\epsilon}\\ i\text{ is
prime}}}\frac{C(n-i)^2}{i}\ll\sum_{\substack{d<n-n^{1-\epsilon}-M \\
d\text{ is
square-free}}}\frac{2^{\omega(d)}}{d}\sum_{\substack{M\leq i\leq
n-n^{1-\epsilon}\\ i\text{ is prime}\\ i\equiv
n\pmod{d}}}\frac{1}{i}
$$
By the Brun-Titchmarsh theorem, we know
$$
|\{M\leq i\leq y:\, i\text{ is prime and }i\equiv
n\pmod{d}\}|\ll\frac{y-M}{\phi(d)\log((y-M)/d)},
$$
provided that $y-M\geq 1.1d$. By the partial summation, for $d\leq (n-n^{1-\epsilon}-M)/1.1$,
\begin{align*}
&\sum_{\substack{M\leq i\leq n-n^{1-\epsilon}\\ i\text{ is prime}\\
i\equiv
n\pmod{d}}}\frac{1}{i}\\\ll&\frac{1}{n-n^{1-\epsilon}}\cdot\frac{n}{\phi(d)\log(n/d)}+\frac1M+\sum_{M<
y<M+1.1d}\frac{2}{y^2}+\sum_{\substack{M+1.1d\leq
y\leq
n-n^{1-\epsilon}}}\frac{1}{\phi(d)y\log(y/d)}\\
\ll&\frac{1}{\phi(d)}+\frac1M+\frac{\log\log(n/d)-\log\log(M/d)}{\phi(d)}.
\end{align*}
If $d\leq\sqrt{M}$, then
$$
\log\log(n/d)-\log\log(M/d)\leq\log\log n-\log\log\sqrt{M}\ll 1.
$$
If $\sqrt{M}<d\leq (n-n^{1-\epsilon}-M)/1.1$, then
$$
\frac{\log\log(n/d)-\log\log(M/d)}{\phi(d)}\ll\frac{\log\log
n}{d^{1-\epsilon}}\ll\frac{1}{M^{1/2-\epsilon}}.
$$
Thus we have
\begin{align}
&\sum_{\substack{d\leq(n-n^{1-\epsilon}-M)/1.1\\ d\text{ is
square-free}}}\frac{2^{\omega(d)}}{d}\sum_{\substack{M\leq i\leq n-n^{1-\epsilon}\\ i\text{ is prime}\\
i\equiv
n\pmod{d}}}\frac{1}{i}\ll\sum_{d<n}\frac{2^{\omega(d)}}{d}\bigg(\frac{1}{M^{1/2-\epsilon}}+\frac{1}{\phi(d)}\bigg)\notag\\
\ll&\frac{1}{M^{1/2-\epsilon}}\sum_{d<n}\frac{1}{d^{1-\epsilon}}+\sum_{d<n}\frac{1}{d^{2-\epsilon}}\ll\frac{n^{\epsilon}}{M^{1/2-\epsilon}}+O(1)=O(1).
\end{align}
On the other hand, we have
\begin{align*}
&\sum_{\substack{(n-n^{1-\epsilon}-M)/1.1<d\leq n-n^{1-\epsilon}-M\\ d\text{ is
square-free}}}\frac{2^{\omega(d)}}{d}\sum_{\substack{M\leq i\leq n-n^{1-\epsilon}\\ i\text{ is prime}\\
i\equiv
n\pmod{d}}}\frac{1}{i}\\
\ll&\sum_{\substack{(n-n^{1-\epsilon}-M)/1.1<d\leq n-n^{1-\epsilon}-M\\ d\text{ is
square-free}}}\frac{2^{\omega(d)}}{n}\sum_{\substack{M\leq i\leq n-n^{1-\epsilon}\\ i\text{ is prime}\\
i=n-d}}1\leq\frac1n\sum_{\substack{i<n\\ i\text{ is prime}}}\tau(n-i),\end{align*}
where $\tau$ be the divisor function and $\tau(d)=2^{\omega(d)}$ if $d$ is square-free. Now
\begin{align*}
\sum_{\substack{i<n\\ i\text{ is prime}}}\tau(n-i)=&\sum_{\substack{i<n\\ i\text{ is prime}}}\sum_{k\mid n-i}1\leq 2
\sum_{\substack{i<n\\ i\text{ is prime}}}\sum_{\substack{k\leq\sqrt{n-i}\\ k\mid n-i}}1\\
\leq&\sum_{k\leq\sqrt{n}}\sum_{\substack{i<n\\ i\text{ is prime}\\ i\equiv n\pmod{k}}}1\ll\sum_{k\leq\sqrt{n}}\frac{n}{\phi(k)\log n}.\end{align*}
We know that (cf. \cite{Sitaramachandra85})
$$
\sum_{k\leq x}\frac{1}{\phi(k)}\ll\log x.
$$
It follows that
\begin{equation}
\sum_{\substack{(n-n^{1-\epsilon}-M)/1.1<d\leq n-n^{1-\epsilon}-M\\ d\text{ is
square-free}}}\frac{2^{\omega(d)}}{d}\sum_{\substack{M\leq i\leq n-n^{1-\epsilon}\\ i\text{ is prime}\\
i\equiv
n\pmod{d}}}\frac{1}{i}\ll 1.
\end{equation}
Thus
$$
\Delta\ll \c^3\log n.
$$
Now
$$
\frac{\E(Y_n^*)^2}{\E(Y_n^*)+\Delta}=\frac{\E(Y_n^*)}{1+\frac{\Delta}{\E(Y_n^*)}}\gg\frac{\c^2\log n}{1+\frac{\c^3\log n}{\c^2\log n}}=\frac{\c^2}{1+\c}\log n.
$$
So we may choose sufficiently large $\c$ such that
$$
\frac{\E(Y_n^*)^2}{\E(Y_n^*)+\Delta}\geq 100\log n.
$$
In view of Lemma \ref{janson}, we have
$$
\P(Y_n^*=0)\leq \P(Y_n^*\leq0.5\E(Y_n^*))\leq\exp\bigg(-\frac{-(0.5\E(Y_n^*))^2}{2(\E(Y_n^*)+\Delta)}\bigg)\leq\exp(-2\log n)=\frac1{n^2}.
$$

Thus choosing a sufficiently large $n_0$, we have
$$
\P(Y_n^*=0\text{ for some }n\geq n_0)\leq\sum_{n\geq n_0}\P(Y_n^*=0)\leq\sum_{n\geq n_0}\frac1{n^2}\leq\frac12.
$$
This completes the proof of Theorem 1.1.

\section{Proof of the Theorem \ref{t2}}
\setcounter{equation}{0} \setcounter{Thm}{0} \setcounter{Lem}{0}
\setcounter{Cor}{0}

Let $c_0$ be the constant appearing in Lemma \ref{jia} and $c_1$ be an another fixed constant with $c_0<c_1<1.$
\begin{Lem}\label{sum} For every $0<\epsilon<1$,there is a
$K=K(\epsilon)$ and an $N_0=N_0(\epsilon)$ such that for $N>N_0$
we can always find a set
$$B\subset [N^{c_0}, 2N^{c_0}]$$
of primes such that
\begin{equation}\label{Bg}
|B|\leq K\log N
\end{equation}
and the set $S=\cP+B$ satisfies
\begin{equation}\label{Sg}
S(x)\geq (1-\epsilon)x \text { for all } N^{c_1}\leq x\leq N.
\end{equation}
\end{Lem}
\begin{proof}
Assume that $N$ is sufficiently large. Suppose that $C^*$ is the
implied constant in (\ref{jiae}). Let
$$K=\max\big\{\frac{10}{c_0C^*}\log\frac{16}{\epsilon^2},20\big\},\quad
M=N^{c_0},\quad I=[M, 2M]\cap \cP,\quad
L=|I|=(1+o(1))\frac{M}{\log M}.$$
Let $B$ be a random subset of
$I$ such that each $x\in I$ is included into $B$ independently
with a probability
$$\varrho=\frac{K\log N}{2L}.$$ And let $t_x$ be the binary random
variable representing the choice of $x$. Clearly
$$
\E(|B|)=\E\bigg(\sum_{x\in I}t_x\bigg)=\varrho L=\frac{K}2\log N.
$$
Now it suffices to show that
\begin{equation}\P((\ref{Sg})\ \text{holds})>\P(|B|>K\log N).
\end{equation}

Apparently,
\begin{align*}
\E(e^{|B|})=&\E\bigg(\exp\bigg(\sum_{x\in I}t_x\bigg)\bigg)=\prod_{x\in I}\E(e^{t_x})=((1-\varrho)\cdot e^0+\varrho\cdot e^1)^L\\
=&(1+\varrho(e-1))^L\leq\exp(\varrho(e-1)L)<\exp\bigg(\frac{K(e-1)}{1.9}\log N\bigg).\end{align*}
Thus from Markov's inequality we get
\begin{align*}
&\P(|B|>K\log N)=\P(e^{|B|}>e^{K\log N})\leq\frac{\E(e^{|B|})}{e^{K\log N}}\\
\leq &\exp\bigg(\frac{K(e-1)}{1.9}\log N-K\log N\bigg)=N^{\frac{K(e-2.9)}{1.9}}<\frac1N.
\end{align*}

Below we only need to show that $\P((\ref{Sg})\ \text{holds})\geq N^{-1}$. Let $\eta=1+\epsilon/2$ and let
$x_j=N/\eta^j$ for $0\leq j\leq (1-c_1)\log N/\log \eta+1$. Clearly, for $x_{j}\leq x\leq x_{j-1}$,
$S(x_j)>(1-\epsilon/2)x_j$ implies that $S(x)>(1-\epsilon)x$. So it suffices to show that
$$
\P(S(x_j)>(1-\epsilon/2)x_j\text{ for all }j)>\frac1N.
$$

Let $$T(x)=x-S(x).$$ Clearly $S(x_j)>(1-\epsilon/2)x_j$ is
equivalent to $T(x_j)<(\epsilon/2)x_j$.

Let
$$
z(n)=\sum_{\substack{ n=p_1+p_2\\ p_1\in I,\ p_2\in\cP}}1
$$
and
$$
Q_j=\big\{M\leq n\leq x_j:\, z(n)<\frac{C^*C(n)M}{3\log^2 n}\big\}
$$
By Lemma \ref{jia}, we have
$$
|Q_j|\leq\sum_{1\leq k\leq x_j/M}|Q_j\cap[kM,(k+1)M]|=\sum_{1\leq
k\leq
x_j/M}O\bigg(\frac{M}{(\log(kM))^2}\bigg)=O\bigg(\frac{x_j}{(\log
M)^2}\bigg).
$$
For a given $n\subseteq[2M,x_j]\setminus Q_j$, we have $n\not\in
S$ only if none of those $z(n)$ primes
$$
\{p_1\in I:\,n-p_1\text{ is prime}\}
$$
is in $B$. The probability of this event is
$$\P(n\not\in S)=(1-\varrho)^{z(n)}\leq\exp(-\varrho z(n)).$$

We have
$$\varrho z(n)\geq \frac{K\log N}{2L}\cdot C^*C(n)\frac{M}{3\log^2 n}\geq\frac{C^*C(n)KM}{6L\log N}\geq\frac{C^*c_0K}{10},$$
since
$$
C(n)\geq \prod_{p\geq 3\text{ is prime}}(1-\frac{1}{(p-1)^2})>0.6.
$$
Hence
\begin{equation}
\P(n\not\in S)\leq \exp(-0.1c_0C^*K).
\end{equation}
Thus the expectation of $T(x_j)$
satisfies
\begin{equation}
\E(T(x_j))\leq\exp(-0.1c_0C^*K)x_j+2M+O(x_j(\log
M)^{-2})\leq2\exp(-0.1c_0C^*K)x_j,
\end{equation}
by noting that $x_j\geq N^{c_1}$. From Markov's inequality, we can infer
\begin{equation}
\P(T(x_j)\geq (\epsilon/2)x_j)\leq
\frac{\E(T(x_j))}{(\epsilon/2)x_j}<\frac{4}{\epsilon}\exp(-0.1c_0C^*K)\leq
1-\frac1\eta,
\end{equation}
since
\begin{equation}
K\geq \frac{10}{c_0C^*}\log\frac{4}{\epsilon(1-\eta^{-1})}.
\end{equation}
So letting $J=\floor{(1-c_1)\log N/\log\eta+1}$, we have
\begin{equation}
\P(S(x_j)>(1-\epsilon/2)x_j\ \text{for all j})\geq
\prod_{j=0}^J \P(S(x_j)>(1-\epsilon/2)x_j)\geq \eta^{-J-1}\geq\frac{1}{\eta^2N^{1-c_1}}\geq\frac{1}{N},
\end{equation}
where $\floor{\alpha}=\max\{z\in\Z:\, z\leq\alpha\}$.
Thus the proof of Lemma 3.1 is complete.
\end{proof}
\begin{Lem}\label{density} For every $\epsilon >0$, let
$K=K(\epsilon)$ and $N_0=N_0(\epsilon)$ be the ones defined in
Lemma \ref{sum}. Then there exists a set $B\subseteq\cP$ such that
for every $x>N_0$, we have $$ B(x)\leq
\frac{2K}{c_0c_1(1-c_1)}\log x
$$
and the sumset $S=\cP+B$ satisfies
$$
S(x)\geq (1-\epsilon)x.
$$\end{Lem}
\begin{proof}
Define
$$ N_{i+1}=[N_i^{1/c_1}]+1$$
for $i\geq 0$. Applying Lemma \ref{sum} to each $N=N_i$, we get a
set
$$B_i\subset [N_i^{c_0},2N_i^{c_0}]$$
with $$|B_i|\leq K\log N_i$$ satisfying that
$$S_i(x)\geq (1-\epsilon)x$$
for any $N_i^{c_1}\leq x\leq N_i$,  where the set $S_i=\cP+B_i$.
We put $$B=\bigcup B_i.$$ Suppose that $x\in[N_i^{c_1},N_i]$. Clearly, for the set $S=\cP+B$, we have
$$S(x)\geq S_i(x)\geq(1-\epsilon)x.$$
Let
$$k=\min\{i:\, N_i^{c_0}>x\text{ and }N_{i-1}^{c_0}\leq x\}.
$$
Note that
$$
\log N_{i-1}\leq c_1\log N_i.
$$
Since $\log N_i$ grows exponentially, we get
$$B(x)\leq \sum_{\substack{0\leq i\leq k}}|B_i|\leq K(\log N_0+\log N_1+\cdots +\log N_k)\leq K\log N_k\sum_{j=0}^kc_1^j\leq\frac{2K\log(x^{1/c_0c_1})}{1-c_1}.$$
\end{proof}
Next we are ready to prove Theorem \ref{t2}.
\begin{proof}[Proof of the Theorem \ref{t2}]

Let
$$
\epsilon_i=\frac1{i+1}.
$$
Let $B_i$ be the set satisfying the requirements of Lemma
\ref{density} for $\epsilon=\epsilon_i$. Let
$N_1=N_0(\epsilon_1)$. Since $w(x)\to\infty$ as $x\to\infty$, for
every $i\geq 2$, let $N_i\geq
\max\{N_0(\epsilon_i),e^{2N_{i-1}/\epsilon_i}\}$  be an integer
such that
$$
\frac{2}{c_0c_1(1-c_1)}\sum_{j=1}^{i+1}K(\epsilon_j)\leq w(x)
$$
for every $x\geq N_i$. Let
$$
\cB=\bigcup_{i=1}^\infty (B_i\cap[N_{i-1},N_{i+1}]).
$$
Note that for any $N_{i}<x\leq N_{i+1}$, we have
\begin{align*}
&|\{n\leq x:\, n=p+b,\ p\in\cP,\ b\in\cB\}|\\
\geq&|\{n\leq x:\, n=p+b,\ p\in\cP,\ b\in B_i\cap[N_{i-1},N_{i+1}]\}|\\
>&|\{n\leq x:\, n=p+b,\ p\in\cP,\ b\in B_i\}|-N_{i-1}\cdot\cP(x)\\
\geq&(1-\epsilon_i)x-N_{i-1}\cdot\frac{2x}{\log x}\geq (1-2\epsilon_i)x.
\end{align*}
So the sumset $\cP+\cB$ has lower density 1. Also, for every
$N_{i}<x\leq N_{i+1}$, we have
\begin{align*}
\cB(x)\leq\sum_{j=1}^{i+1} |B_j\cap[1,x]|
\leq\bigg(\frac{2}{c_0c_1(1-c_1)}\sum_{j=1}^{i+1}K(\epsilon_j)\bigg)\log
x\leq w(x)\log x.
\end{align*}
\end{proof}
\begin{Ack} The second author thanks his colleague, Dr. Rongli Liu, for her helpful discussions on Vu's proof.
\end{Ack}

\end{document}